\documentclass[a4paper,12pt]{article}
\usepackage[english]{babel}
\usepackage{hyperref}
\usepackage{indentfirst}
\usepackage{amsfonts,amssymb,amsthm,amsmath}
\theoremstyle{definition}

\theoremstyle{theorem}
\newtheorem{Th}{Theorem}
\theoremstyle{lemma}
\newtheorem{Lm}{Lemma}
\theoremstyle{plain}

\theoremstyle{remark}

\theoremstyle{plain}
\newtheorem{coll}{Corollary}
\newcounter{tmp}
\newcommand{\ir}{[0,T]\times \Sigma_{d}}
\begin{document}
\title{Extremal shift rule for continuous-time zero-sum Markov games}
\author{Yurii Averboukh\footnote{Krasovskii Institute of Mathematics and Mechanics UrB RAS, ayv@imm.uran.ru}}
\maketitle
\begin{abstract}
In the paper we consider the controlled continuous-time Markov chain describing the interacting particles system with the finite number of types. The system is controlled by two players with the opposite purposes. The limiting game as the number of particles tends to infinity is a zero-sum differential game. Krasovskii--Subbotin  extremal shift provides the optimal strategy in the limiting game. The main result of the paper is the near optimality of the Krasovskii--Subbotin extremal shift rule for the original Markov game.

\vspace{3pt}\noindent\textbf{Keywords:} continuous time Markov games, differential games, extremal shift rule, control with guide strategies.
\end{abstract}
\section{Introduction}

The paper is devoted to the construction of  near optimal strategies for zero-sum two players continuous-time Markov game based on deterministic game. The term `Markov game' is used for a Markov chain with the Kolmogorov matrix depending on controls of players. These games are also called continuous-time stochastic games.
First continuous-time Markov games were studied by Zachrisson \cite{Zachrisson}. The information of recent progress in the  theory of continuous-time Markov games can be found in \cite{Neyman}, \cite{Levy} and references therein.

We consider the case when the continuous-time Markov chain describes the interacting particle system. The interacting particle system converges to the deterministic system as the number of particles tends to infinity~\cite{Kol},~\cite{Kol_book} (see also~\cite{Darling_Norris},~\cite{Benaim_le_boduak}).  The value function of the controlled Markov chain converges to the value function of the limiting control system  \cite{Kol} (see also corresponding result for discrete-time systems in~\cite{Gast_et_al}). This result is extended to the  case of zero-sum games as well as to the case of  nonzero-sum games~\cite{Kol}. If the nonanticipative strategy is optimal for  differential game then it is near optimal for the Markov game~\cite{Kol}. However the  nonanticipative strategies require the knowledge of the  control of the second player. Often this information is inaccessible and the player has only the information about current position. In this case one can use feedback strategies or control with guide strategies.

Control with guide strategies were proposed by Krasovskii and Subbotin to construct the solution of deterministic differential game under informational disturbances~\cite{NN_PDG_en}. Note that the feedback strategies do not provide the stable solution of the differential game. If the player uses control with guide strategy, then the control is formed stepwise, and the player has a model of the system and she uses this model to choose an appropriate control using extremal shift rule. The value function is achieved in the limit when the time between control corrections tends to zero. In the original work by Krasovskii and Subbotin the motion of the model is governed by the system that is a copy of the original system and the motion of the original system is close to the motion of the model. Therefore the model can be called guide. Note that formally control with guide strategy is a strategy with memory. However, it suffices to storage only finite number of vectors. Additionally, the player should use computer to obtain the state of the guide at the time of control correction.

Control with guide strategies realizing the extremal shift were used for the differential games without Lipschitz continuity of the dynamics in \cite{kriazh} and for  the games  governed by delay differential equations in~\cite{kras_delay},~\cite{luk_plaks}.  Krasovskii and Kotelnikova proposed the stochastic control with guide strategies \cite{a4}--\cite{a6}. In that case the real motion of the deterministic system is close to the auxiliary stochastic process generated by optimal control for the stochastic differential game.
The Nash equilibrium for two-player game in the class of  control with guide strategies was constructed via extremal shift in \cite{Averboukh_jcds}.

In this paper we let the  player use the control with guide strategy realizing extremal shift rule in the Markov game. We assume that the motion of the guide is given by  the limiting deterministic differential game. We estimate the expectation of the distance between the Markov chain and the motions of the model (guide). This leads to the estimate between the outcome of the player in the Markov game and the value function of the limiting differential game.

The paper is organized as follows. In preliminary Section \ref{sect_prel} we describe the  Markov game describing the interacting particle system and the limiting deterministic differential game. In Section \ref{sect_cgs} we give the explicit definition of control with guide strategies and formulate the main results. Section \ref{sect_prop} is devoted to a property of transition probabilities. In Section \ref{sect_estima} we estimate the expectation of  distance between the Markov chain and the deterministic guide. Section \ref{sect_proof} provides the proofs of the statements formulated in Section \ref{sect_cgs}.

\section{Preliminaries}\label{sect_prel}
We consider the system of finite number particles. Each particle can be of  type~$i$, $i\in \{1,\ldots,d\}$.  The type of each particle is a random variable governed by a Markov chain. To specify this chain consider the Kolmogorov matrix  $Q(t,x,u,v)=(Q_{ij}(t,x,u,v))_{i,j=1}^d$. That means that the elements of matrix $Q(t,x,u,v)$ satisfy the following properties
\begin{itemize}
  \item $Q_{ij}(t,x,u,v)\geq 0$ for $i\neq j$;
  \item
  \begin{equation}\label{Kolmogorov}
  Q_{ii}(t,x,u,v)=-\sum_{j\neq i}Q_{ij}(t,x,u,v).
  \end{equation}
\end{itemize}
Here $$t\in [0,T],\ \ x\in  \Sigma_d=\{(x_1,\ldots, x_n):x_i\geq 0, x_1+\ldots+x_n=1\},\ \ u\in U,v\in V.$$ Suppose that $U$ and $V$ are compact sets. The variables $u$ and $v$ are controlled by the first and the second players respectively. Below we assume that $x=(x_1,\ldots,x_n)$ is a row-vector.
Additionally we assume that
\begin{itemize}
 \item $Q$ is a continuous function of its variable;
 \item for any $t$, $u$  and $v$ the function $x\mapsto Q(t,x,u,v)$ is Lipschitz continuous;
 \item for any $t\in [0,T]$, $\xi,x\in\mathbb{R}^n$ the following equality holds true
 \begin{equation}\label{isaacs}
   \min_{u\in U}\max_{v\in V}\langle \xi,xQ(t,x,u,v)\rangle=\max_{v\in V}\min_{u\in U}\langle \xi,xQ(t,x,u,v)\rangle
 \end{equation}
\end{itemize}
Condition (\ref{isaacs}) is an analog of well-known Isaacs condition.

For a fixed parameters $x\in\mathbb{R}^d$, $u\in{U}$, and $v\in{V}$ the type of each particle is determined by the Markov chain with the generator
$$(Q(t,x,u,v)f)_i=\sum_{j\neq i}Q_{ij}(t,x,u,v)(f_j-f_i), \ \ f=(f_1,\ldots,f_d). $$
The another way to specify the Markov chain is the  Kolmogorov forward equation
$$\frac{d}{dt}P(s,t,x)=P(s,t,x)Q(t,x,u,v). $$ Here $P(s,t,x)=(P_{ij}(s,t,x))_{ij=1}^d$ is the matrix of the transition probabilities.

Now we  consider the controlled mean-field interacting particle system (see \cite{Kol}). Let $n_i$ be a number of particles of the type $i$. The vector $N=(n_1,\ldots,n_d)\in \mathbb{Z}_+^d$ is the state of the system consisting of $|N|=n_1+\ldots+n_d$ particles. For $i\neq j$ and a vector $N=(n_1,\ldots,n_d)$ denote by $N^{[ij]}$ the vector obtained from $N$ by removing one particle of  type $i$ and adding one particle of  type $j$ i.e. we replace the $i$-th coordinate with $n_i-1$ and the $j$-th coordinate with $n_j+1$.  
The mean-field interacting particle system 
is a Markov chain with the generator
$$L_t^h[u,v]f(N)=\sum_{i,j=1}^d n_iQ_{ij}(t,N/|N|,u(t),v(t))[f(N^{[ij]})-f(N)]. $$
The purpose of the first (respectively, second) player is to minimize (respectively, maximize) the expectation of $\sigma(N/|N|)$.

Denote the inverse number of particles by $h=1/|N|$. Normalizing the states of the interacting particle system
we get the generator (see \cite{Kol})
\begin{equation}\label{generator}
L_t^h[u,v]f(N/|N|)=\sum_{i,j=1}^d \frac{1}{h}\frac{n_i}{|N|}Q_{ij}(t,N/|N|,u(t),v(t))\left[f\left(\frac{N^{[ij]}}{|N|}\right)-f\left(\frac{N}{|N|}\right)\right]. \end{equation}
Denote  the vector  $N/|N|$ by $x=(x_1,\ldots,x_d)$. 
Thus, we have that
$$L_t^h[u,v]f(x)=\sum_{i,j=1}^d \frac{1}{h}x_iQ_{ij}(t,x,u(t),v(t))[f(x-he^i+he^j)-f(x)]. $$ Here $e^i$ is the $i$-th coordinate vector.
The vector $x$ belongs to the set $$\Sigma_d^h=\{(x_1,\ldots,x_d):x_i\in h\mathbb{Z},\ \ x_1+\ldots+x_d=1\}\subset \Sigma_d. $$

Further, let $\mathcal{U}_{\rm det}[s]$ (respectively, $\mathcal{V}_{\rm det}[s]$) denote the set of deterministic controls of the first (respectively, second) player on $[s,T]$, i.e.
$$\mathcal{U}_{\rm det}[s]=\{u:[s,T]\rightarrow U\mbox{ measurable}\}, \ \ \mathcal{V}_{\rm det}[s]=\{v:[s,T]\rightarrow V\mbox{ measurable}\}. $$

Let $(\Omega,\mathcal{F},\{\mathcal{F}_t\},P)$ be a filtered probability space. Extending the definition given in \cite[p. 135]{fleming_soner} to the stochastic game case, we say that the pair of stochastic processes $u$ and $v$ on $[s,T]$ is an admissible pair of controls if
\begin{enumerate}
  \item $u(t)\in U$, $v(t)\in V$;
  \item  the processes $u$ and $v$ are progressive measurable;
  \item for any $y\in \Sigma^h_d$ there exists an unique $\{\mathcal{F}_t\}_{t\in [s,T]}$-adapted c\`{a}dl\`{a}g stochastic process $X^h(t,s,y,u,v)$ taking values in $\Sigma^h_d$, starting at $y$ at time $s$ and satisfying the following condition
      \begin{equation}\label{dynkin}
      \mathbb{E}_{sy}^hf(X^h(t,s,y,u,v))-f(y)=\int_s^t\mathbb{E}_{sy}^h L_{t}^h[u(\tau),v(\tau)]f(X^h(\tau,s,y,u,v))d\tau.
 \end{equation}
\end{enumerate}
Here $\mathbb{E}_{sy}^h$ denotes the conditional expectation of corresponding stochastic processes.

The purposes of the players can be reformulated in the following way. The first (respectively, second) player wishes to minimize (respectively, maximize) the value
$$\mathbb{E}_{sy}^h\sigma(X_h(T,s,y,u,v)). $$

Let $\mathcal{U}^h[s]$ be a set of stochastic processes $u$ taking values in $U$ such that the pair $(u,v)$ is admissible for any $v\in\mathcal{V}_{\rm det}[s]$. Analogously, let $\mathcal{V}^h[s]$ be a set of stochastic processes $v$ taking values in $V$ such that the pair $(u,v)$ is admissible for any $u\in\mathcal{U}_{\rm det}[s]$.

Denote by $P_{sy}^h(A)$ the conditional probability of the event $A$ under condition that the Markov chain corresponding to the parameter $h$ starts at $y$ at time $s$, i.e.
$$P^h_{sy}(A)=\mathbb{E}_{sy}^h\mathbf{1}_{A}$$
Further, let
 $p^h(s,y,t,z,u,v)$ denote the transition probability i.e.
$$p^h(s,y,t,z,u,v)=P_{sy}^h(X^h(t,s,y,u,v)=z)=\mathbb{E}_{sy}^h\mathbf{1}_{\{z\}}(X^h(t,s,y,u,v)). $$
 The substituting $\mathbf{1}_{\{z\}}$ for $f$ in (\ref{generator}) and (\ref{dynkin}) gives that
\begin{multline}\label{tran_prob}p^h(s,y,t,z,v)= p^h(s,y,s,z,u,v)\\
+\frac{1}{h}\int_{s}^t\mathbb{E}^h_{sy}\sum_{i,j=1}^dX_{h,i}(\tau,s,y,u,v)Q_{ij}(\tau,X_h(\tau,s,y,u,v),u(\tau),v(\tau)) \\\cdot[\mathbf{1}_z(X_h(\tau,s,y,u,v)-he^i+he^j)-\mathbf{1}_z(X_h(\tau,s,y,u,v))]d\tau. \end{multline}
Here $X_{h,i}(\tau,s,y,u,v)$ denotes the $i$-th component of $X_h(\tau,s,y,u,v)$.

Recall, see \cite{Kol}, that if $h\rightarrow 0$, then  the  generator $L_t^h[u,v]$ converges to the generator
\begin{equation*}\begin{split}\Lambda_t[u,v]f(x)=&\sum_{i=1}^d\sum_{j\neq i}x_iQ_{ij}(t,x,u(t),v(t))\left[\frac{\partial f}{\partial x_j}(x)-\frac{\partial f}{\partial x_i}(x)\right]\\=&\sum_{k=1}^d\sum_{i\neq k}[x_iQ_{ik}(t,x,u(t),v(t))-x_kQ_{ki}(t,x,u(t),v(t))]\frac{\partial f}{\partial x_k}(x).\end{split}\end{equation*}
For  controls  $u\in\mathcal{U}_{\rm det}[s]$ and $v\in\mathcal{V}_{\rm det}[s]$ the deterministic evolution generated by the $\Lambda_t[u(t),v(t)]$ is described by the equation
\begin{multline}\label{deter_evol}
\frac{d }{d t}f_t(x)=\sum_{k=1}^d\sum_{i\neq k}[x_iQ_{ik}(t,x,u(t),v(t))-x_kQ_{ki}(t,x,u(t),v(t))]\frac{\partial f_t}{\partial x_k}(x),\\ f_s(x)=f(x).
\end{multline}
 Here the function $f_t(y)$ is equal to $f(x(t))$ when $x(s)=y$.
The characteristics of (\ref{deter_evol}) solve the ODEs
\begin{equation*}\begin{split}\frac{d}{dt}{x}_k(t)=&\sum_{i\neq k}[x_i(t)Q_{ik}(t,x(t),u(t),v(t))-x_k(t)Q_{ki}(t,x(t),u(t),v(t))]\\=&\sum_{i=1}^d x_i(t)Q_{ik}(t,x(t),u(t),v(t)). \end{split}\end{equation*}
One can rewrite this equation in the vector form
\begin{equation}\label{char_eq}
  \frac{d}{dt}{x}(t)=x(t)Q(t,x(t),u(t),v(t)),\ \ t\in [0,T], \ \ x(t)\in\mathbb{R}^n,\ \ u(t)\in U,\ \ v(t)\in V.
\end{equation}
For given $u\in\mathcal{U}_{\rm det}[s]$, $v\in \mathcal{V}_{\rm det}[s]$ denote the solution of initial value problem for~(\ref{char_eq}) and condition $x(s)=y$ by $x(\cdot,s,y,u,v)$.
Consider the deterministic zero-sum game with the dynamics given by (\ref{char_eq}) and terminal payoff equal to $\sigma(x(T,s,y,u,v))$. This game has a value that is a continuous function of the position. Denote it by ${\rm Val}(s,y)$. Recall (see \cite{Subb_book}) that the function ${\rm Val}(s,y)$ is a minimax (viscosity) solution of the Hamilton--Jacobi PDE
\begin{equation}\label{HJ_eq}
  \frac{\partial W}{\partial t}+H(t,x,\nabla W)=0,\ \ W(T,x)=\sigma(x).
\end{equation}
Here the Hamiltonian $H$ is defined by the rule
$$H(t,x,\xi)=\min_{u\in U}\max_{v\in V}\langle \xi,xQ(t,x,u,v)\rangle. $$

\section{Control with guide strategies}\label{sect_cgs}
In this section we introduce the control with guide strategies for the Markov game. It is assumed that the control is formed stepwise and the player has an information about the current state of the system i.e. the vector $x$ is known. Additionally, we assume that the player can evaluate the expected state and the player's control depends on current state of the system and on the evaluated state. This evaluation is called guide. At each time of control correction the player computes the value of the guide and the control that is used up to the next time of control correction.

Formally (see \cite{Subb_chen}), control with guide strategy of player 1 is a triple $\mathfrak{u}=(u(t,x,w),\psi_1(t_+,t,x,w),\chi_1(s,y))$. Here the function $u(t,x,w)$ is equal to the control implemented after  time $t$ if at time $t$ the state of the system is $x$ and the state of the guide is $w$. The function $\psi_1(t_+,t,x,w)$ determines the state of the guide at time $t_+$ under the condition that at time $t$ the state of the system is $x$ and the state of the guide is $w$. The function $\chi_1$ initializes the guide i.e. $\chi_1(s,y)$ is the state of the guide in the initial position $(s,y)$.

We use the control with guide strategies for Markov game with the generator $L_t^h$. Here we assume that $h>0$ is fixed. Let $(s,y)$ be an initial position, $s\in [0,T]$ $y\in\Sigma_d^h$. Assume that player 1 chooses the control with guide strategy $\mathfrak{u}$ and the partition $\Delta=\{t_k\}_{k=0}^m$ of the time interval $[s,T]$; whereas player 2 chooses the control $v\in\mathcal{V}^h[s]$. This control can be also formed stepwise using some second player's control with guide strategy.

We say that the stochastic process $\mathcal{X}_1^h[\cdot,s,y,\mathfrak{u},\Delta,v]$ is generated by strategy $\mathfrak{u}$, partition $\Delta$ and the second player's control $v$ if for $t\in [t_k,t_{k+1})$
$\mathcal{X}_1^h[t,s,y,\mathfrak{u},\Delta,v]=X^h(t,t_k,x_k,u_k,v), $ where
\begin{itemize}
  \item $x_0=y$, $w_0=\chi_1(t_0,x_0)$, $u_0=u(t_0,x_0,w_0)$;
  \item for $k=\overline{1,r}$ $x_k=X^h(t_k,t_{k-1},x_{k-1},u_{k-1},v)$, $w_k=\psi_1(t_k,t_{k-1},x_{k-1},w_{k-1})$, $u_k=u(t_k,x_k,w_k)$.
\end{itemize}
Note that even though the state of the guide $w_k$ is determined by the deterministic function it depends on the random variable $x_{k-1}$. Thus, $w_k$ is a random variable.

Below we define the first player's control with guide strategy that realizes the  extremal shift rule (see \cite{NN_PDG_en}). Let $\varphi$ be a supersolution of equation (\ref{HJ_eq}). That means (see \cite{Subb_book}) that for any $(t_*,x_*)\in \ir$, $t_+>t_*$ and $v_*\in V$ there exists a solution $\zeta_1(\cdot,t_+,t_*,x_*,v_*)$ of differential inclusion
$$\dot{\zeta}_1(t)\in {\rm co}\{\zeta_1(t)Q(t,\zeta_1(t),u,v_*):u\in U\} $$ satisfying conditions
$\zeta_1(t_*,t_+,t_*,x_*,v_*)=x_*$ and $\varphi(t_+,\zeta_1(t_+,t_+,t_*,x_*,v_*))\leq \varphi(t_*,x_*)$.

Define the control with guide strategy $\hat{\mathfrak{u}}=(\hat{u},\hat{\psi}_1,\hat{\chi}_1)$ by the following rules. If $t_*,t_+\in [0,T]$, $t_+>t_*$, $x_*,w_*\in \Sigma_d$, then
choose $u_*$, $v_*$ by the rules
\begin{equation}\label{u_star_def}
    \min_{u\in U}\max_{v\in V}\langle x_*-w_*,x_*Q(t_*,x_*,u,v)\rangle=\max_{v\in V}\langle x_*-w_*,x_*Q(t_*,x_*,u_*,v)\rangle,
\end{equation}
\begin{equation}\label{v_star_def}
    \max_{v\in V}\min_{u\in U}\langle x_*-w_*,x_*Q(t_*,x_*,u,v)\rangle=\min_{u\in U}\langle x_*-w_*,x_*Q(t_*,x_*,u,v_*)\rangle.
\end{equation}
Put
\begin{list}{\rm (u\arabic{tmp})}{\usecounter{tmp}}
\item\label{str_def_1} $\hat{u}(t_*,x_*,w_*)=u_*$,
\item\label{str_def_2}  $\hat{\psi}_1(t_+,t_*,x_*,w_*)=\zeta_1(t_+,t_+,t_*,w_*,v_*)$,
\item\label{str_def_3}  $\hat{\chi}_1(s,y)=y$.
\end{list}
Note that if the first player uses the strategy $\hat{\mathfrak{u}}$ in the differential game with the dynamics given by (\ref{char_eq}) then she guarantees the limit outcome not greater then $\varphi$ (see~\cite{NN_PDG_en},~\cite{Subb_book}). If additionally $\varphi={\rm Val}$, then the strategy $\hat{\mathfrak{u}}$ is optimal in the deterministic game.

The main result of the paper is the following.
\begin{Th}\label{th_kras_sub_prob} Assume that $\sigma$ is Lipschitz continuous with a constant $R$, and the function $\varphi$ is a supersolution of (\ref{HJ_eq}).
If the first player uses the control with guide strategy $\hat{\mathfrak{u}}$ determined by (u\ref{str_def_1})--(u\ref{str_def_3}) for the function $\varphi$ then
\begin{list}{\rm (\roman{tmp})}{\usecounter{tmp}}
\item 
\begin{multline*}\lim_{\delta\downarrow 0}\sup\{\mathbb{E}_{sy}^h(\sigma(\mathcal{X}_1^h[T,s,y,\hat{\mathfrak{u}},\Delta,v])):d(\Delta)\leq \delta,v\in \mathcal{V}^h[s]\}\\\leq \varphi(s,y)+R\sqrt{Dh}.\end{multline*}
\item \begin{multline*}
\lim_{\delta\downarrow 0}\sup\Bigl\{P_{sy}^h\Bigl(\sigma(\mathcal{X}_1^h[T,s,y,\hat{\mathfrak{u}},\Delta,v])\geq \varphi(s,y)+R\sqrt[3]{Dh}\Bigr):\\d(\Delta)\leq\delta,\ \ v\in\mathcal{V}^h[s]\Bigl\} \leq \sqrt[3]{Dh}.
\end{multline*}
\end{list}
Here $D$ is a constant not dependent on $\varphi$ and $\sigma$.
\end{Th}
The theorem is proved in Section \ref{sect_proof}.

Now let us consider the case when the second player uses control with guide strategies. The control with guide strategy of the second player is a triple $\mathfrak{v}=(v(t,x,w),\psi_2(t_+,t,x,w),\chi_2(s,y))$. Here $w$ denotes the state of the second player's guide. The control in this case is formed also stepwise. If $(s,y)$ is an initial position, $\Delta$ is a partition of time interval $[s,T]$ and $u\in\mathcal{U}^h[s]$ is a control of player 1 then denote by $\mathcal{X}_2^h[\cdot,s,y,\mathfrak{v},\Delta,u]$ the corresponding stochastic process.

Let $\omega$ be a subsolution of equation (\ref{HJ_eq}). That means (see \cite{Subb_book}) that for any $(t_*,x_*)\in \ir$, $t_+>t_*$ and $u^*$ there exists a trajectory $\zeta_2(\cdot,t_+,t_*,x_*,u^*)$ of the differential inclusion
$$\dot{\zeta}_2(t)\in {\rm co}\{\zeta_2(t)Q(t,\zeta_2(t),u^*,v):v\in V\} $$ satisfying conditions
$\zeta_2(t_*,t_+,t_*,x_*,u^*)=x_*$ and $\omega(t_+,\zeta_2(t_+,t_+,t_*,x_*,u^*))\geq \omega(t_*,x_*)$.

Define the strategy $\hat{\mathfrak{v}}$ by the following rule. If $(t_*,x_*)$ is a position, $t_+>t_*$ and $w_*\in\Sigma_d$ is a state of the guide then choose $v^*$ and $u^*$ by the rules
$$
  \min_{v\in V}\max_{u\in U}\langle x_*-w_*,x_*Q(t_*,x_*,u,v)\rangle=\max_{u\in U}\langle x_*-w_*,x_*Q(t_*,x_*,u,v^*)\rangle,
$$
$$
 \max_{u\in U}\min_{v\in V}\langle x_*-w_*,x_*Q(t_*,x_*,u,v)\rangle=\min_{v\in V}\langle x_*-w_*,x_*Q(t_*,x_*,u^*,v)\rangle.
$$
Put
\begin{list}{\rm (v\arabic{tmp})}{\usecounter{tmp}}
\item $v(t_*,x_*,w_*)=v^*$,
\item $\psi_2(t_+,t_*,x_*,w_*)=\zeta_2(t_+,t_+,t_*,x_*,u^*)$
\item $\chi_2(s,y)=y$.
\end{list}
\begin{coll}\label{coll_second}
If the second player uses the control with guide strategy $\hat{\mathfrak{v}}$ determined by (v1)--(v3) for the function $\omega$ that is a subsolution of (\ref{HJ_eq}), then
\begin{list}{\rm (\roman{tmp})}{\usecounter{tmp}}
\item \begin{multline*}\lim_{\delta\downarrow 0}\inf\{\mathbb{E}_{sy}^h(\sigma(\mathcal{X}_1^h[T,s,y,\hat{\mathfrak{u}},\Delta,v])):d(\Delta)\leq \delta,u\in\mathcal{U}^h[s]\}\\\geq \omega(s,y)-R\sqrt{Dh}.\end{multline*}
\item \begin{multline*}
\lim_{\delta\downarrow 0}\sup\Bigl\{P_{sy}^h\Bigl(\sigma(\mathcal{X}_2^h[T,s,y,\hat{\mathfrak{v}},\Delta,u])\leq \omega(s,y)-R\sqrt[3]{Dh}\Bigr):\\d(\Delta)\leq\delta,\ \ u\in\mathcal{U}^h[s]\Bigl\} \leq \sqrt[3]{Dh}.
\end{multline*}\end{list}
\end{coll}
The corollary is also proved in Section \ref{sect_proof}.

\section{Properties of transition probabilities}\label{sect_prop}
Now we prove the following.
\begin{Lm}\label{lm_trans_prob} There exists a function $\alpha^h(\delta)$ such that $\alpha^h(\delta)\rightarrow 0$ as $\delta\rightarrow 0$ and for any $t_*,t_+\in [0,T]$, $\xi,\eta\in \Sigma_d$, $\xi=(\xi_1,\ldots,\xi_d)$, $\bar{u}\in U$, $\bar{v}\in \mathcal{V}^h[t_*]$
\begin{enumerate}
  \item if $\eta=\xi$, then \begin{multline*}p^h(t_*,\xi,t_+,\eta,\bar{u},\bar{v})\\\leq 1+\frac{1}{h}\sum_{k=1}^d\int_{t_*}^{t_+}\int_{V} \xi_kQ_{kk}(t_*,\xi,\bar{u},v)\nu_\tau (dv)d\tau+\alpha^h(t_+-t_*)\cdot(t_+-t_*);\end{multline*}
  \item if $\eta=\xi-h e^i+he^j$, then
  $$p^h(t_*,\xi,t_+,\eta,u,v)\leq \frac{1}{h}\int_{t_*}^{t_+}\int_{V} \xi_iQ_{ij}(t_*,\xi,\bar{u},v)\nu_\tau(dv) d\tau+\alpha^h(t_+-t_*)\cdot(t_+-t_*); $$
  \item if $\eta\neq \xi$ and $\eta\neq \xi-h e^i+he^j$, then
  $$p^h(t_*,\xi,t_+,\eta,u,v)\leq \alpha^h(t_+-t_*)\cdot(t_+-t_*); $$
\end{enumerate}
Here $\nu_\tau$ is a measure on $V$ depending on $t_*,t_+$, $\xi$, $\eta$, $\bar{u}$ and $\bar{v}$.
\end{Lm}
\begin{proof}
First denote
\begin{equation}\label{K_def}
K=\sup\{|Q_{ij}(t,x,u,v)|:i,j=\overline{1,d},\ \ t\in[0,T],\ \ x\in \Sigma_d,\ \ u\in U,\ \ v\in Q\}.
\end{equation}
Note that for any $x\in\Sigma_d$, $t\in [0,T]$, $u\in U$, $v\in V$ the following estimates hold true
\begin{equation}\label{sigma_d_property}
\|x\|\leq \sqrt{d},\ \ \left|\sum_{i=1}^n x_iQ_{ij}(t,x,u,v)\right|\leq K,\ \ \|xQ(t,x,u,v)\|\leq K\sqrt{d}.
\end{equation}
Further, let $\gamma(\delta)$ be a common modulus of continuity with respect to $t$ of the functions $Q_{ij}$ i.e. for all $i$, $j$, $t',t''\in [0,T]$, $x\in\Sigma_d$, $u\in U$, $v\in Q$
\begin{equation}\label{gamma_def}
  |Q_{ij}(t',x,u,v)-Q_{ij}(t'',x,u,v)|\leq \gamma(t''-t')
\end{equation} and $\gamma(\delta)\rightarrow 0$ as $\delta\rightarrow 0$.
From (\ref{tran_prob}) and (\ref{sigma_d_property}) we obtain that
\begin{equation}\label{prob_der_estima}
p^h(t_*,\xi,t,\eta,u,v)\leq p^h(t_*,\xi,t_*,\eta,u,v)+\frac{2Kd}{h}(t-t_*).
\end{equation}

Further, for a given control $\bar{v}\in\mathcal{V}^h[t_*]$ let $\mathbb{E}^h_{t_*\xi;\tau x}$ denote the expectation under conditions $X^h(t_*,t_*,\xi,\bar{u},\bar{v})=\xi$, and $X^h(\tau,t_*,\xi,\bar{u},\bar{v})=x$.

We have that $$\mathbb{E}_{t_*\xi}^hf=\sum_{x\in\Sigma_d^h}\mathbb{E}^h_{t_*\xi;\tau x}f\cdot p^h(t_*,\xi,\tau,x,\bar{u},\bar{v}). $$  From this and (\ref{tran_prob}) we get
\begin{equation*}\begin{split}
  p^h(t_*,\xi,t,\eta,\bar{u}&,\bar{v})=p(t_*,\xi,t,\eta,\bar{u},\bar{v})\\+\frac{1}{h}\int_{t_*}^t\sum_{x\in\Sigma_d^h}\mathbb{E}^h_{t_*\xi;\tau x} &\sum_{i,j=1}^d x_iQ_{i,j}(\tau,x,\bar{u},\bar{v}(\tau))[\mathbf{1}_\eta(x-he^i+he^j)-\mathbf{1}_\eta(x)]\\\cdot &p(t_*,\xi,\tau,x,\bar{u},\bar{v})d\tau  \leq
  p(t_*,\xi,t,\eta,\bar{u},\bar{v})\\
  +\frac{1}{h}\int_{t_*}^t&\sum_{x\in\Sigma_d^h}\mathbb{E}^h_{t_*\xi;\tau x} \sum_{i,j=1}^d x_iQ_{i,j}(\tau,x,\bar{u},\bar{v}(\tau))[\mathbf{1}_\eta(x-he^i+he^j)-\mathbf{1}_\eta(x)]\\\cdot &p(t_*,\xi,t_*,x,\bar{u},\bar{v})d\tau+\frac{2K^2d^2}{h}(t-t_*)^2.
\end{split}\end{equation*}
We have that $p(t_*,\xi,t_*,x,\bar{u},\bar{v})=1$ for $x=\xi$ and $p(t_*,\xi,t_*,x,\bar{u},\bar{v})=0$ for $x\neq \xi$. Thus,
\begin{equation*}\begin{split}
  p^h(t_*,\xi,t,\eta,\bar{u},\bar{v}&)\leq p(t_*,\xi,t,\eta,\bar{u},\bar{v}) \\
  +\frac{1}{h}\int_{t_*}^t\mathbb{E}^h_{t_*\xi;\tau \xi} \sum_{i,j=1}^d &\xi_iQ_{i,j}(\tau,\xi,\bar{u},\bar{v}(\tau))[\mathbf{1}_\eta(\xi-he^i+he^j)-\mathbf{1}_\eta(\xi)] +\frac{2K^2d^2}{h}(t-t_*)^2\\ \leq
  p(t_*,\xi,t,\eta,\bar{u},&\bar{v}) \\
  +\frac{1}{h}\int_{t_*}^t\mathbb{E}^h_{t_*\xi;\tau \xi} &\sum_{i,j=1}^d \xi_iQ_{i,j}(t_*,\xi,\bar{u},\bar{v}(\tau))[\mathbf{1}_\eta(\xi-he^i+he^j)-\mathbf{1}_\eta(\xi)]d\tau\\ &+\frac{2K^2d^2}{h}(t-t_*)^2+\frac{2d}{h}\gamma(t-t_*)\cdot(t-t_*).
\end{split}\end{equation*}
There exists a measure $\nu_\tau$ on V such that
$$\mathbb{E}_{t_*\xi;\tau \xi}Q_{ij}(t_*,\xi,\bar{u},\bar{v}(\tau))=\int_VQ_{ij}(t_*,\xi,\bar{u},v)\nu_\tau(dv). $$
Consequently,
\begin{multline}\label{trans_prob_estima}
  p^h(t_*,\xi,t,\eta,\bar{u},\bar{v})\leq p(t_*,\xi,t,\eta,\bar{u},\bar{v})
\\+\frac{1}{h}\int_{t_*}^t\int_{V} \sum_{i,j=1}^d \xi_iQ_{i,j}(t_*,\xi,\bar{u},v)[\mathbf{1}_\eta(\xi-he^i+he^j)-\mathbf{1}_\eta(\xi)]\nu_\tau(dv)d\tau\\ +\alpha(t-t_*)\cdot (t-t_*).
\end{multline}
Here we denote $$\alpha(\delta)=\frac{2K^2d^2}{h}(\delta)2+\frac{2d}{h}\gamma(\delta). $$
From (\ref{trans_prob_estima}) the second and  third statements of the Lemma follows. To derive the first statement use the property of Kolmogorov matrixes (\ref{Kolmogorov}). We have that
\begin{multline*}
  p^h(t_*,\xi,t,\xi,\bar{u},\bar{v})\leq p(t_*,\xi,t,\eta,\bar{u},\bar{v})
\\-\frac{1}{h}\int_{t_*}^t\int_{V} \sum_{i=1}^d\sum_{j\neq i} \xi_iQ_{i,j}(t_*,\xi,\bar{u},v)\nu_\tau(dv)d\tau +\alpha(t-t_*)\cdot (t-t_*)\\=
p(t_*,\xi,t,\eta,\bar{u},\bar{v})
+\frac{1}{h}\int_{t_*}^t\int_{V} \sum_{i=1}^d \xi_iQ_{i,i}(t_*,\xi,\bar{u},v)\nu_\tau(dv)d\tau +\alpha(t-t_*)\cdot (t-t_*).
\end{multline*}

\end{proof}

\section{Key estimate}\label{sect_estima}
This section provides the estimate of the distance between the controlled Markov chain and the guide. This estimate is an analog of \cite[Lemma 2.3.1]{NN_PDG_en}.
\begin{Lm}\label{lm_kras_subb}
There exist constants $\beta,C>0$, and a function $\varkappa^h(\delta)$ such that $\varkappa^h(\delta)\rightarrow 0$ as $\delta\rightarrow 0$ and the following property holds true.

\noindent If\begin{enumerate}
\item $(t,x)\in \ir^h$, $w_*\in\Sigma_d$, $t_+>t_*$,
\item the controls $u_*$  $v_*$ are chosen by rules (\ref{u_star_def}) and (\ref{v_star_def}) respectively,
\item $w_+=\zeta_1(t_+,t_+,t_*,w_*,v_*)$,
\end{enumerate} then for any $v\in\mathcal{V}^h[t_*]$
\begin{multline*}\mathbb{E}_{t_*x_*}^h(\|\mathcal{X}(t_+,t_*,x_*,u_*,v)-w_+\|^2)\\\leq (1+\beta(t_+-t_*))\|x_*-w_*\|+Ch(t_+-t_*)+\varkappa^h(t_+-t_*)\cdot(t-t_*). \end{multline*}
\end{Lm}
\begin{proof}
Denote the $i$-th component of vector $x_*$ by $x_{*i}$.

We have that
\begin{equation}\label{expectation}
\mathbb{E}_{t_*x_*}^h(\|\mathcal{X}(t_+,t_*,x_*,u_*,v)-w_+\|^2)=\sum_{z\in\Sigma_d^h}\|z-w_+\|^2p(t_*,x_*,t_+,z,u_*,v).
\end{equation}
Further,
\begin{equation*}\begin{split}\|z-w_+\|^2=\|(z-x_*)&+(x_*-w_*)+(w_*-w_+)\|^2\\
=\|x_*-w_*&\|^2+2\langle x_*-w_*,z-x_*\rangle\\-2\langle &x_*-w_*,w_+-w_*\rangle +\|z-x_*\|^2+\|w_+-w_*\|^2.
 \end{split}\end{equation*}
It follows from (\ref{sigma_d_property}) that
 \begin{equation}\label{w_estimate}
\left\|\frac{d}{dt}\zeta_1(t_+,t,t_*,w_*,v_*)\right\|\leq K\sqrt{d},\ \  \|w_+-w_*\|^2\leq K^2d(t_+-t_*)^2.
 \end{equation}
From Lemma \ref{lm_trans_prob} it follows that
\begin{equation}\label{z_estima}\begin{split}
  \sum_{z\in\Sigma_d^h}\|z-x_*\|^2p(t_*,x_*,t_+,z,u_*,v&)\\\leq \sum_{i=1}^d\sum_{j\neq i}\|-he^i+he^j&\|^2\frac{1}{h}\int_{t_*}^{t_+}\int_VQ_{ij}(t_*,x_*,u_*,v)\nu_\tau(dv)d\tau\\ &+2d^3\alpha^h(t_+-t_*)\cdot(t_+-t_*) \\
  \leq 2hd^2K(t_+-t_*)+2&d^3\alpha^h(t_+-t_*)\cdot(t_+-t_*).
\end{split}\end{equation}
For simplicity denote $\zeta_*(t)=\zeta_1(t,t_+,t_*,w_*,u,v_*)$. We have that for each $t$ there exists a probability $\mu_t$ on $V$ such that
$$\frac{d\zeta_*}{dt}(t)=\int_{u\in U}\zeta_*(t)Q(t,\zeta_*(t),u,v_*)\mu_t(du). $$ Therefore,
\begin{equation}\label{inner_transform}\begin{split}
    \sum_{z\in\Sigma_d^h}\langle x_*-w_*,&w_+-w_*\rangle p(t_*,x_*,t_+,z,u_*,v)\\=\Bigl\langle x_*-&w_*, \int_{t_*}^{t_+}\int_{u\in U}\zeta_*(t)Q(t,\zeta_*(t),u,v_*)\mu_t(du)dt\Bigr\rangle.
\end{split}\end{equation}

Define
\begin{equation}\label{rho_def}\begin{split}\varrho(\delta&)\triangleq\sup\{|y''Q(t'',y'',u,v)-y'Q(t',y',u,v)|:\\ &t',t''\in [0,T], \ \ y',y''\in \Sigma_d,\ \ u\in U,\ \ v\in V,\ \ |t'-t''|\leq \delta,\ \ \|y'-y''\|\leq \delta K\sqrt{d}\}. \end{split}\end{equation} We have that $\varrho(\delta)\rightarrow 0$, as $\delta\rightarrow 0$.
From (\ref{w_estimate}), (\ref{inner_transform}), and (\ref{rho_def}) it follows that
\begin{equation}\label{cross_estima_guide}\begin{split}
  &\sum_{z\in\Sigma_d^h}\langle x_*-w_*,w_+-w_*\rangle p(t_*,x_*,t_+,z,u_*,v)\\
  &\geq \left\langle x_*-w_*,\int_{t_*}^{t_+}\int_{u\in U}w_*Q(t_*,w_*,u,v_*)\mu_t(du)dt\right\rangle-\sqrt{2d}\varrho(t_+-t_*)\cdot(t_+-t_*).
\end{split}\end{equation}
Using  Lemma \ref{lm_trans_prob} one more time we get the inequality
\begin{equation}\label{cross_estima_state_pre}\begin{split}
  \sum_{z\in\Sigma_d^h}\langle x_*-w_*,z-x_*\rangle p(t_*,x_*,t_+,z,u_*,v&)\\ \leq
  \sum_{i=1}^d\sum_{j\neq i}\langle x_*-w_*,-he^i+he^j&\rangle\frac{1}{h}\int_{t_*}^{t_+}\int_V x_{*i}Q_{ij}(t_*,x_*,u_*,v)\nu_t(dv)dt\\&+2d^3\alpha^h(t_+-t_*)\cdot(t_+-t_*).
\end{split}\end{equation}
The first term in the right-hand side of (\ref{cross_estima_state_pre}) can be transformed as follows. Denote for simplicity
$$\widehat{Q}_{ij}=\int_{t_*}^{t_+}\int_V Q_{ij}(t_*,x_*,u_*,v)\nu_t(dv)dt .$$ Note that $\widehat{Q}=(\widehat{Q}_{ij})_{i,j=1}^d$ is a Kolmogorov matrix. That means that $$-\sum_{j\neq i}\widehat{Q}_{ij}=\widehat{Q}_{ii}. $$ We have that
\begin{equation*}\begin{split}
  \sum_{i=1}^d\sum_{j\neq i}(-he^i+he^j&)\frac{1}{h}\int_{t_*}^{t_+}\int _V x_{*i}Q_{ij}(t_*,x_*,u_*,v)\nu_t(dv)dt \\
  &=\sum_{i=1}^d\sum_{j\neq i}e^j x_{*,i}\widehat{Q}_{ij}-\sum_{i=1}^d x_{*,i}e^i\sum_{j\neq i} \widehat{Q}_{ij}=\\
  &=\sum_{i=1}^d \sum_{j=1}^d e^j x_{*,i}\widehat{Q}_{ij}=\sum_{j=1}^d\left[\sum_{i=1}^d x_{*,i}\widehat{Q}_{ij}\right]e^j=x_*\widehat{Q}.
\end{split}\end{equation*}
This and (\ref{cross_estima_state_pre}) yield the estimate
\begin{equation}\label{cross_estima_state_fin}\begin{split}
  \sum_{z\in\Sigma_d^h}&\langle x_*-w_*,z-x_*\rangle p(t_*,x_*,t_+,z,u_*,v)\\ \leq
&\left\langle x_*-w_*,\int_{t_*}^{t_+}\int_V x_*Q(t_*,x_*,u_*,v)\nu_t(dv)dt\right\rangle +2d^3\alpha^h(t_+-t_*)\cdot(t_+-t_*).
\end{split}\end{equation}
Substituting (\ref{w_estimate})--(\ref{cross_estima_guide}), (\ref{cross_estima_state_fin}) in (\ref{expectation}) we get the estimate
\begin{equation}\label{expec_stima_pre}\begin{split}
  \mathbb{E}_{t_*x_*}^h(\|\mathcal{X}(t_+,&t_*,x_*,u_*,v)-w_+\|^2)\leq  \|x_*-w_*\|^2\\
  &+ 2\left\langle x_*-w_*,\int_{t_*}^{t_+}\int_V x_*Q(t_*,x_*,u_*,v)\nu_t(dv)dt\right\rangle\\
  &-2\left\langle x_*-w_*,\int_{t_*}^{t_+}\int_{u\in U} w_*Q(t_*,w_*,u,v_*)\mu_t(du)dt\right\rangle\\
  &+2Kd^2 h(t_+-t_*)+ (6d^3\alpha^h(t_+-t_*)+\sqrt{2d}\varrho(t_+-t_*))\cdot(t_+-t_*).
\end{split}\end{equation}

Let $L$ be a Lipschitz constant of the function $y\mapsto yQ(t,y,u,v)$ i.e. for all $y',y''\in\Sigma_d$, $t\in [0,T]$, $u\in U$, $v\in Q$
$$\|y'Q(t,y',u,v)-y''Q(t,y'',u,v)\|\leq L\|y'-y''\|. $$

 We have that
\begin{equation*}\begin{split}
  2\Bigl\langle x_*-w_*,\int_{t_*}^{t_+}&\int_V x_*Q(t_*,x_*,u_*,v)\nu_t(dv)dt\Bigr\rangle\\&-2\Bigl\langle x_*-w_*,\int_{t_*}^{t_+}\int_{u\in U} w_*Q(t_*,w_*,u,v_*)\mu_t(du)dt\Bigr\rangle \\ \leq
  2\int_{t_*}^{t_+}\int_{u\in U}&\int_{v\in V}\Bigl [\Bigl\langle x_*-w_*, x_*Q(t_*,x_*,u_*,v)\Bigr\rangle\\-\Bigl\langle x_*-&w_*, x_*Q(t_*,x_*,u,v_*)\Bigr\rangle\Bigr]\nu_t(dv)\mu_t(du)dt\\&+2L\|x_*-w_*\|^2(t_+-t_*).
\end{split}\end{equation*}
The choice of $u_*$ and $v_*$ gives that for all $u\in U$ and $v\in V$
$$\langle x_*-w_*, x_*Q(t_*,x_*,u_*,v)\rangle\leq \langle x_*-w_*, x_*Q(t_*,x_*,u,v_*)\rangle. $$ Consequently, we get the estimate
\begin{equation}\label{extr_shift_fin}\begin{split}
  2\Bigl\langle x_*-w_*,\int_{t_*}^{t_+}\int_V x_*Q(t_*,x_*,u_*,v)\nu_t&(dv)dt\Bigr\rangle\\-2\Bigl\langle x_*-w_*,\int_{t_*}^{t_+}\int_{u\in U} w_*Q(t_*,w_*,&u,v_*)\mu_t(du)dt\Bigr\rangle \\
  &\leq 2L\|x_*-w_*\|^2(t_+-t_*).
\end{split}\end{equation} From (\ref{expec_stima_pre}) and (\ref{extr_shift_fin}) the conclusion of the Lemma follows for
$$\beta=2L,\ \ C=2d^2K, \ \ \varkappa^h(\delta)=6d^3\alpha^h(\delta)+\sqrt{2d}\varrho(\delta). $$\end{proof}

\section{Near Optimal Strategies}\label{sect_proof}
In this section we prove Theorem \ref{th_kras_sub_prob} and Corollary \ref{coll_second}.
\begin{proof}[Proof of Theorem \ref{th_kras_sub_prob}]
Let $v\in\mathcal{V}^h[s]$ be a control of the second player. Consider a partition $\Delta=\{t_k\}_{k=1}^m$ of the time interval $[s,T]$. If $x_0,x_1,\ldots,x_m$ are  vectors, $x_0=y$ then denote by $\hat{p}^h_r(x_1,\ldots,x_r,\Delta)$ the probability of the event $\mathcal{X}_1^h[t_k,s,y,\hat{\mathfrak{u}},\Delta,v]=x_k$ for $k=\overline{1,r}$.
Define vectors $w_0,\ldots,w_m$ recursively in the following way. Put
 \begin{equation}\label{w_0_def}
 w_0\triangleq\hat{\chi}_1(s,y)=y,
\end{equation}
for $k>0$ put
  \begin{equation}\label{w_k_def}
  w_k\triangleq\hat{\psi}_1(t_k,t_{k-1},x_{k-1},w_{k-1}).
\end{equation} If $w_0,\ldots,w_m$ are defined by rules (\ref{w_0_def}), (\ref{w_k_def}) and $r\in \overline{1,n}$ we write $$(w_0,\ldots,w_r)=g_r(x_0,\ldots,x_{r-1},\Delta). $$ In addition, put $g_0(\Delta)\triangleq y$.

Below we use the transformation $G(\cdot,\mathcal{X}^h_1[\cdot,s,y,\hat{\mathfrak{u}},\Delta,v])$ of the stochastic $\mathcal{X}^h_1[\cdot,s,y,\hat{\mathfrak{u}},\Delta,v]$ defined in the following way. If $x_i$ are values of $\mathcal{X}^h_1[t_i,s,y,\hat{\mathfrak{u}},\Delta,v]$, $i=0,\ldots,r$, and $(w_0,\ldots,w_r)=g_r(x_0,\ldots,x_{r-1},\Delta)$, then we put $$G(t_r,\mathcal{X}^h_1[\cdot,s,y,\hat{\mathfrak{u}},\Delta,v])\triangleq w_r. $$ Generally, the stochastic process $G(\cdot,\mathcal{X}^h_1[\cdot,s,y,\hat{\mathfrak{u}},\Delta,v])$ is non-Markov.

Further, if $u_i=\hat{u}(t_i,x_i,w_i)$, $i=0,\ldots,r$, and $(w_0,\ldots,w_r)=g_r(x_0,\ldots,x_{r-1},\Delta)$, we write $\varsigma_r(x_0,\ldots,x_r,\Delta)\triangleq u_r$.

We have that for any $r\in \overline{1,m}$
\begin{equation}\label{first_expectation_estima}\begin{split}
               \mathbb{E}_{sy}^h(\|\mathcal{X}^h_1[t_r,s,y,\hat{\mathfrak{u}},\Delta,&v]- G(t_r,\mathcal{X}^h_1[\cdot,s,y,\hat{\mathfrak{u}},\Delta,v])\|^2) \\=
               \sum_{x_1,\ldots,x_r}\|x_r-g_r(x_0,&\ldots,x_{r-1},\Delta)\|^2\hat{p}_r(x_0,\ldots,x_{r},\Delta)\\=
               \sum_{x_1,\ldots,x_{r-1}}\hat{p}_{r-1}(x_0,&\ldots,x_{r-1},\Delta) \cdot\sum_{x_r}\|x_r-g_r(x_0,\ldots,x_{r-1},\Delta)\|^2 \\\cdot &P_{t_{r-1}x_{r-1}}^h(X(t_r,t_{r-1},x_{r-1},\varsigma_{r-1}(x_0,\ldots,x_{r-1}),v)=x_r).
\end{split}\end{equation}
By Lemma \ref{lm_kras_subb} we have that
\begin{equation*}\begin{split}
  \sum_{x_r}\|x_r-g_r(x_1,\ldots,x_{r-1},&\Delta))\|^2 \\\cdot P_{t_{r-1}x_{r-1}}^h(X(t_r,&t_{r-1},x_{r-1},\varsigma_{r-1}(x_0,\ldots,x_{r-1}),v)=x_r) \\
  \leq (1+\beta(t_{r}-t_{r-1}))\|x_{r-1}&-g_{r-1}(x_0,\ldots,x_{r-2},\Delta)\|^2\\ +Ch\cdot(t_r-&t_{r-1})+\varkappa^h(t_r-t_{r-1})\cdot(t_r-t_{r-1}).
\end{split}\end{equation*}
From this and (\ref{first_expectation_estima}) it follows that
\begin{equation}\label{ineq_r}\begin{split}
               \mathbb{E}_{sy}^h(\|\mathcal{X}^h_1[t_r,s,y,\hat{\mathfrak{u}},&\Delta,v]- G(t_r,\mathcal{X}^h_1[\cdot,s,y,\hat{\mathfrak{u}},\Delta,v])\|^2)\\ \leq
               (1+&\beta(t_{r}-t_{r-1}))\mathbb{E}_{sy}^h(\|x_{r-1}-g_{r-1}(x_0,\ldots,x_{r-2})\|^2)\\&+ Ch\cdot(t_r-t_{r-1})+\varkappa^h(t_r-t_{r-1})\cdot(t_r-t_{r-1}).
               \end{split}\end{equation}
Applying this inequality recursively we get
\begin{equation*}\begin{split}
               \mathbb{E}_{sy}^h(\|\mathcal{X}^h_1[\cdot,s,y,\hat{\mathfrak{u}},\Delta,v&]- G(T,\mathcal{X}^h_1[\cdot,s,y,\hat{\mathfrak{u}},\Delta,v])\|^2)\\ \leq
               \exp(&\beta (T-s))\mathbb{E}_{sy}^h(\|x_0-g_0(\Delta)\|^2)\\&+Ch\cdot (T-s)+\varkappa^h(d(\Delta))\cdot(T-s).
               \end{split}\end{equation*}
Taking into account the equality $x_0=y=g_0(\Delta)$ we conclude that
\begin{equation}\label{final_estima}
               \mathbb{E}_{sy}^h(\|\mathcal{X}^h_1[T,s,y,\hat{\mathfrak{u}},\Delta,v]- G(T,\mathcal{X}^h_1[\cdot,s,y,\hat{\mathfrak{u}},\Delta,v])\|^2)\\ \leq \epsilon(h,d(\Delta)).
               \end{equation}
Here we denote $$\epsilon(h,\delta)\triangleq Dh+T\varkappa^h(\delta),\ \ D\triangleq CT.$$ Note that for any $h$
\begin{equation}\label{epsilon_convergence}
\epsilon(h,\delta)\rightarrow Dh,\mbox{ as } \delta\rightarrow 0.
\end{equation}
From (\ref{final_estima}) and Jensen's inequality we get
\begin{equation}\label{expect_estima}
\mathbb{E}_{sy}^h(\|\mathcal{X}^h_1[T,s,y,\hat{\mathfrak{u}},\Delta,v]- G(T,\mathcal{X}^h_1[\cdot,s,y,\hat{\mathfrak{u}},\Delta,v])\|)\leq\sqrt{\epsilon(d(\Delta),h)}.
\end{equation}

By construction of control with guide strategy $\hat{\mathfrak{u}}$
\begin{equation*}\begin{split}\varphi(s,y)=\varphi(t_0,g_0(\Delta))&\geq\varphi(t_1,g_1(x_0,\Delta)) \geq\ldots\\&\geq\varphi(t_m,g_m(x_0,\ldots,x_{m0-1},\Delta))=\sigma(g_m(x_0,\ldots,x_{m0-1},\Delta)). \end{split}\end{equation*}
Hence,
 \begin{equation}\label{sigma_varphi}
 \sigma(G(T,\mathcal{X}^h_1[\cdot,s,y,\hat{\mathfrak{u}},\Delta,v]))\leq \varphi(s,y).
\end{equation}

Since $\sigma$ is Lipschitz continuous with the constant $R$ we have that for any partition $\Delta$ and second player's control $v$
$$\sigma(\mathcal{X}^h_1[T,s,y,\hat{\mathfrak{u}},\Delta,v])\leq \varphi(s,y)+R\|\mathcal{X}^h_1[T,s,y,\hat{\mathfrak{u}},\Delta,v]- G(T,\mathcal{X}^h_1[\cdot,s,y,\hat{\mathfrak{u}},\Delta,v])\|. $$
This and (\ref{expect_estima}) give the inequality
$$\mathbb{E}_{sy}^h\sigma(\mathcal{X}^h_1[T,s,y,\hat{\mathfrak{u}},\Delta,v])\leq \varphi(s,y)+R\sqrt{\epsilon(d(\Delta),h)}. $$
Passing to the limit as $d(\Delta)\rightarrow 0$ and taking into account the property $\epsilon(\delta,h)\rightarrow Dh$, as $\delta\rightarrow 0$, (see \ref{epsilon_convergence}) we obtain the first statement of the Theorem.

Now let us prove the second statement of the Theorem.
Using Markov inequality and (\ref{final_estima}) we get
\begin{equation*}\begin{split}P\bigl(\|&\mathcal{X}^h_1[T,s,y,\hat{\mathfrak{u}},\Delta,v]- G(T,\mathcal{X}^h_1[\cdot,s,y,\hat{\mathfrak{u}},\Delta,v])\|\geq[\epsilon(h,d(\Delta))]^{1/3}\bigr)\\ &=P\bigl(\|\mathcal{X}^h_1[T,s,y,\hat{\mathfrak{u}},\Delta,v]- G(T,\mathcal{X}^h_1[\cdot,s,y,\hat{\mathfrak{u}},\Delta,v])\|^2\geq [\epsilon(h,d(\Delta))]^{2/3}\bigr)\\ &\leq
\frac{\mathbb{E}_{sy}^h(\|\mathcal{X}^h_1[T,s,y,\hat{\mathfrak{u}},\Delta,v]- G(T,\mathcal{X}^h_1[\cdot,s,y,\hat{\mathfrak{u}},\Delta,v])\|^2)}{[\epsilon(h,d(\Delta))]^{2/3}}
\leq \sqrt[3]{\epsilon(h,d(\Delta))}. \end{split}\end{equation*}
Lipschitz continuity of the function $\sigma$ and (\ref{sigma_varphi}) yield the following inclusion
\begin{equation*}\begin{split}
  \{\sigma(\mathcal{X}_1^h[T,s,y,\hat{\mathfrak{u}},\Delta,v]&)\geq\varphi(s,y)+R[\epsilon(h,d(\Delta))]^{1/3}\}\subset \\\bigl\{\|\mathcal{X}^h_1[T,s,&y,\hat{\mathfrak{u}},\Delta,v]- G(T,\mathcal{X}^h_1[\cdot,s,y,\hat{\mathfrak{u}},\Delta,v])\|\geq[\epsilon(h,d(\Delta))]^{1/3}\bigr\}.
\end{split}\end{equation*}
Finally, for any partition $\Delta$ and any second player's control $v\in\mathcal{V}^h[s]$ we have that
$$P\{\sigma(\mathcal{X}_1^h[T,s,y,\hat{\mathfrak{u}},\Delta,v])\geq\varphi(s,y)+ R[\epsilon(h,d(\Delta))]^{1/3}\}\leq [\epsilon(h,d(\Delta))]^{1/3}. $$
From this the second statement of the Theorem follows.

\end{proof}
To prove  Corollary \ref{coll_second} it suffices to replace the payoff function with $-\sigma$ and interchange the players.

\section{Conclusion}
In the paper we applied the deterministic strategy that is optimal for deterministic zero-sum game to the Markov game describing interacting particle system. We showed that it is near optimal. We considered control with guide strategy. This strategy requires computer to storage and compute a finite dimensional vector that is an evaluation of the current position. The question  whether there exists  an optimal for differential game feedback deterministic strategy that is near optimal for Markov game is open.

We restricted our attention to the Markov game describing the interacting particle systems. The extensions of  the results of the paper to the general case is the theme of future works.

The author would like to thank Vassili Kolokoltsov for insightful discussions.


\begin{thebibliography}{99}
\bibitem{Zachrisson} Zachrisson LE. Markov games. In: Dresher M, Shapley LS, Tucker AW (eds) Advances in game theory. Princeton University Press, Princeton (1964) pp. 211--253
\bibitem{Neyman} Neyman A. Continuous-time stochastic games. DP \#616, Center for the Study of Rationality, Hebrew University, Jerusalem (2012)
\bibitem{Levy} Levy Y.  Continuous-Time Stochastic Games of Fixed Duration. Dynamic Games and Applications
 3 (2013) pp.  279--312
\bibitem{Kol} Kolokoltsov VN. Nonlinear Markov Games on a Finite State Space (Mean-field and
Binary Interactions). International Journal of Statistics and Probability  1, no. 1 (2012) pp. 77--91
\bibitem{Kol_book} Kolokoltsov VN. Nonlinear Markov process and kinetic equations. \textit{Cambridge University Press. Cambridge} (2010)
\bibitem{Darling_Norris}  Darling RWR,  Norris JR. Differential equation approximations for
Markov chains. Probability Surveys  5 (2008) pp. 37--79
\bibitem{Benaim_le_boduak} Bena\"{i}m M,  Le Boudec J-Y.  A class of mean field interaction models for computer and communication systems. Performance Evaluation 65 (2008) pp. 823--838
\bibitem{Gast_et_al}Gast N, Gaujal B,  Le Boudec J-Y.  Mean field for Markov Decision Processes: from Discrete to Continuous
Optimization. INRIA report No. 7239 (2010)

\bibitem{NN_PDG_en} Krasovskii NN, Subbotin AI.  Game-Theoretical Control
Problems. \textit{Springer, New  York} (1988)

\bibitem{kriazh}Kriazhimskii AV. On stable position control in differential games. Journal of Applied Mathematics and Mechanics  42, no 6 (1978) pp. 1055--1060

\bibitem{kras_delay} Krasovskii NN, Kotelnikova AN. Stochastic guide for a time-delay object in a positional differential game. Proceedings of the Steklov Institute of Mathematics  277, Issue 1 Supplement (2012) pp 145-151

\bibitem{luk_plaks}  Lukoyanov NYu,  Plaksin AR. Finite-dimensional modeling guides in time-delay systems. Trudy Instituta Matematiki i Mekhaniki UrO RAN 19, No. 1 (2013) pp. 182--195 (in Russian)

\bibitem{a4}	Krasovskii NN, Kotelnikova AN.	 An approach-evasion differential game: Stochastic guide. Proceedings of the Steklov Institute of Mathematics.  269, Issue~1 Supplement (2010) pp. 191--213

\bibitem{a5}	Krasovskii NN, Kotelnikova AN. Unification of differential games, generalized solutions of the Hamilton-Jacobi equations, and a stochastic guide. Differential Equations  45, Issue 11 (2009) pp 1653--1668

\bibitem{a6}	Krasovskii NN, Kotelnikova AN. On a differential interception game. Proceedings of the Steklov Institute of Mathematics 268, Issue 1 (2010) pp. 161--206

\bibitem{Averboukh_jcds} Averboukh Yu. Universal Nash Equilibrium Strategies for Differential Games. Journal of Dynamical and Control Systems (in press). arXiv:1306.2297


\bibitem{fleming_soner} Fleming WH,  Soner HM.  Controlled Markov Processes and Viscosity Solutions. \textit{Springer, New York} (2006)

\bibitem{Subb_book} Subbotin AI.  Generalized solutions of first-order PDEs. The dynamical perspective. 
\textit{Birkha\"{u}ser, Boston} (1995)




\bibitem{Subb_chen} Subbotin AI, Chentsov AG. Optimization of guarantee in control problems. \textit{Nauka, Moscow} (1981, in Russian)

\end{thebibliography}
\end{document}